\documentclass[letterpaper,10pt]{article}

\usepackage[margin=0.5in]{geometry}

\usepackage{fullpage}
\usepackage{enumerate}
\usepackage{authblk}
\usepackage[colorinlistoftodos]{todonotes}
\usepackage{verbatim}
\usepackage{section, amsthm, textcase, setspace, amssymb, lineno, amsmath, amssymb, amsfonts, latexsym, fancyhdr, longtable, ulem, mathtools}
\usepackage{epsfig, graphicx, pstricks,pst-grad,pst-text,tikz, tkz-berge,colortbl}
\usepackage{epsf}
\usepackage{graphicx, color}
\usepackage{float}
\usepackage[rflt]{floatflt}
\usepackage{amsfonts}
\usepackage{latexsym,enumitem}
\usetikzlibrary{fit,matrix,positioning}
\usepackage{pdflscape}
\usetikzlibrary{decorations.pathreplacing}
\usepackage{kbordermatrix}
\usepackage[most]{tcolorbox}
\usepackage{lipsum}

\usepackage{mathrsfs}

\definecolor{darkgreen}{rgb}{0, 0.5, 0}

\newtheorem{theorem}{Theorem}
\newtheorem{lemma}{Lemma}
\newtheorem{corollary}{Corollary}

\newtheorem{definition}{Definition}

\newtheorem{Ex}{Example}

\newtheorem*{theorem*}{Theorem}
\newtheorem{remark}{Remark}

\newcommand\addvmargin[1]{
  \node[fit=(current bounding box),inner ysep=#1,inner xsep=0]{};}
  
\newcommand{\ind}{{\rm ind \hspace{.1cm}}}

\newcommand\scalemath[2]{\scalebox{#1}{\mbox{\ensuremath{\displaystyle #2}}}}

\begin{document}

\title{The index and spectrum of Lie poset algebras of type B, C, and D}

\author[*]{Vincent E. Coll, Jr.}
\author[*]{Nicholas W. Mayers}

\affil[*]{Department of Mathematics, Lehigh University, Bethlehem, PA, 18015}

\maketitle

\begin{abstract}
\noindent
In this paper, we define posets of types B, C, and D. These posets encode the matrix forms of certain Lie algebras which lie between the algebras of upper-triangular and diagonal matrices. Interestingly, such type-B, C, and D Lie poset algebras can be related to Reiner's notion of a parset. Our primary concern is the index and spectral theories of type-B, C, and D Lie poset algebras. For an important restricted class, we develop combinatorial index formulas and, in particular, characterize posets corresponding to Frobenius Lie algebras.  In this latter case we show that the spectrum is binary; that is, consists of an equal number of 0's and 1's.
\end{abstract}



\section{Introduction}
The study of ``Lie poset algebras" was initiated by Coll and Gerstenhaber in \textbf{\cite{CG}}, where the deformation theory of such algebras was investigated. The authors define Lie poset algebras as subalgebras of $A_{n-1}=\mathfrak{sl}(n)$ which lie between the subalgebras of upper-triangular and diagonal matrices; we will refer to such Lie subalgebras of $\mathfrak{sl}(n)$ as \textit{type-A Lie poset algebras}. In \textbf{\cite{CG}}, the authors suggest a way in which to extend the notion of Lie poset algebra to the other classical families of Lie algebras. Interestingly, the resulting Lie algebras can be related to the notion of a ``parset" as defined by Reiner (see Remark~\ref{rem:reiner} and cf. \textbf{\cite{Reiner1,Reiner2}}). Following the suggestion of \textbf{\cite{CG}}, we define posets which encode the matrix forms of such Lie poset algebras and, here, initiate an investigation into their index and spectral theories.

\bigskip
Formally, the index of a Lie algebra $\mathfrak{g}$ is defined as 
\[\ind \mathfrak{g}=\min_{F\in \mathfrak{g^*}} \dim  (\ker (B_F)),\]

\noindent where $B_F$ is the skew-symmetric \textit{Kirillov form} defined by $B_F(x,y)=F([x,y])$, for all $x,y\in\mathfrak{g}$. Of particular interest are those Lie algebras which have index zero, and are called \textit{Frobenius}.\footnote{Frobenius algebras are of  special interest in deformation and quantum group theory stemming from their connection with the classical Yang-Baxter equation (see \textbf{\cite{G1,G2}}).} A functional $F\in\mathfrak{g}^*$ for which $\dim(\ker(B_F))=\ind\mathfrak{g}=0$ is likewise called \textit{Frobenius}. Given a Frobenius Lie algebra $\mathfrak{g}$ and a Frobenius functional $F\in\mathfrak{g}^*$, the map $\mathfrak{g}\to\mathfrak{g}^*$ defined by $x\mapsto B_F(x,-)$ is an isomorphism. The inverse image of $F$ under this isomorphism, denoted $\widehat{F}$, is called a \textit{principal element} of $\mathfrak{g}$ (see \textbf{\cite{Diatta,Prin}}). In \textbf{\cite{Ooms}}, Ooms shows that the eigenvalues (and multiplicities) of $ad(\widehat{F})=[\widehat{F},-]:\mathfrak{g}\to\mathfrak{g}$ do not depend on the choice of principal element $\widehat{F}$. It follows that the spectrum of $ad(\widehat{F})$ is an invariant of $\mathfrak{g}$, which we call the \textit{spectrum} of $\mathfrak{g}$ (see \textbf{\cite{specD, specAB, unbroken}}).

Recently, there has been motivation to determine combinatorial index formulas for certain families of Lie algebras. Families for which such formulas have been found include seaweed algebras and type-A Lie poset algebras (see \textbf{\cite{CHM,Coll2,CM,DK,Elash,Joseph,Panyushev1,Panyushev2,Panyushev3}}). In this article, we consider the analogues of type-A Lie poset algebras in the other classical types: $B_{k}=\mathfrak{so}(2k+1)$, $C_{k}=\mathfrak{sp}(2k)$, and $D_{k}=\mathfrak{so}(2k)$; such algebras are called \textit{type-B, C, and D Lie poset algebras}, respectively. We find that these Lie poset algebras are encoded by certain posets whose underlying sets are of the form $\{-n,\hdots,-1,0,1,\hdots,n\}$ in type B and of the form $\{-n,\hdots,-1,1,\hdots,n\}$ in types C and D. Furthermore, we fully develop the index and spectral theories of type-B, C, and D Lie poset algebras whose underlying posets have the property that there are no relations between pairs of positive integers and no relations between pairs of negative integers. In particular, for this important base case, we develop combinatorial index formulas which rely on an associated planar graph -- called a \textit{relation graph} (see Theorem~\ref{thm:indform}). These formulas allow us to fully characterize Frobenius, type B, C, and D Lie poset algebras in this case \textup(see Theorem~\ref{thm:FrobC}\textup). This classification leads to the realization that the spectrum of such algebras is \textit{binary}; that is, consists of an equal number of 0's and 1's (see Theorem~\ref{thm:spec}).

The organization of this paper is as follows. Section~\ref{sec:poset} sets the combinatorial definitions and notation needed from the theory of posets. In Sections~\ref{sec:lieposet} and~\ref{sec:BCDpos} we formally introduce type-B, C, and D Lie poset algebras and the posets which encode them. Sections~\ref{sec:indexform} and~\ref{sec:spec} deal with the index and spectral theories of types-B, C, and D Lie poset algebras. Finally, in a short epilogue, we compare the (less complicated) type-A case with the type-B, C, and D cases developed here.

\section{Posets}\label{sec:poset}

A \textit{finite poset} $(\mathcal{P}, \preceq_{\mathcal{P}})$ consists of a finite set $\mathcal{P}$ together with a binary relation $\preceq_{\mathcal{P}}$ which is reflexive, anti-symmetric, and transitive. When no confusion will arise, we simply denote a poset $(\mathcal{P}, \preceq_{\mathcal{P}})$ by $\mathcal{P}$, and $\preceq_{\mathcal{P}}$ by $\preceq$. Throughout, we let $\le$ denote the natural ordering on $\mathbb{Z}$. Two posets $\mathcal{P}$ and $\mathcal{Q}$ are \textit{isomorphic} if there exists an order-preserving bijection $\mathcal{P}\to\mathcal{Q}$. 

\begin{remark}
If $|\mathcal{P}|=n$, then there exists a poset $(\{1,\hdots,n\},\le')$ and $f:(\mathcal{P},\preceq_{\mathcal{P}})\to(\{1,\hdots,n\},\le')$ such that $\le'\subset\le$ and $f$ is an isomorphism. Thus, given a poset $(\mathcal{P},\preceq_{\mathcal{P}})$ with $|\mathcal{P}|=n$, unless stated otherwise, we assume that $\mathcal{P}=\{1,\hdots,n\}$ and $\preceq_{\mathcal{P}}\subset\le$.
\end{remark}

If $x,y\in\mathcal{P}$ such that $x\preceq y$ and there does not exist $z\in \mathcal{P}$ satisfying $x,y\neq z$ and $x\preceq z\preceq y$, then $x\preceq y$ is a \textit{covering relation}.  Covering relations are used to define a visual representation of $\mathcal{P}$ called the \textit{Hasse diagram} -- a graph whose vertices correspond to elements of $\mathcal{P}$ and whose edges correspond to covering relations (see, for example, Figure~\ref{fig:poset}).

\begin{Ex}\label{ex:poset}
Let $\mathcal{P}=\{1,2,3,4\}$ with $1\preceq 2\preceq 3,4$. In Figure~\ref{fig:poset} we illustrate the Hasse diagram of $\mathcal{P}$.

\begin{figure}[H]
$$\begin{tikzpicture}
	\node (1) at (0, 0) [circle, draw = black, fill = black, inner sep = 0.5mm, label=left:{1}]{};
	\node (2) at (0, 1)[circle, draw = black, fill = black, inner sep = 0.5mm, label=left:{2}] {};
	\node (3) at (-0.5, 2) [circle, draw = black, fill = black, inner sep = 0.5mm, label=left:{3}] {};
	\node (4) at (0.5, 2) [circle, draw = black, fill = black, inner sep = 0.5mm, label=right:{4}] {};
    \draw (1)--(2);
    \draw (2)--(3);
    \draw (2)--(4);
    \addvmargin{1mm}
\end{tikzpicture}$$
\caption{Hasse diagram of a poset}\label{fig:poset}
\end{figure}
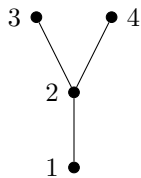
\end{Ex}

Given a subset $S\subset\mathcal{P}$, the \textit{induced subposet generated by $S$} is the poset $\mathcal{P}_S$ on $S$, where, for $x,y\in S$, $x\preceq_{\mathcal{P}_S}y$ if and only if $x\preceq_{\mathcal{P}}y$. A totally ordered subset $S\subset\mathcal{P}$ is called a \textit{chain}. Finally, we define the dual of $\mathcal{P}$, denoted $\mathcal{P}^*$, to be the poset on the same set as $\mathcal{P}$ with $j\preceq_{\mathcal{P}^*}i$ if and only $i\preceq_{\mathcal{P}}j$.

\section{Lie poset algebras}\label{sec:lieposet}

Let $\mathcal{P}$ be a finite poset and \textbf{k} be an algebraically closed field of characteristic zero, which we may take to be the complex numbers. The (associative) \textit{incidence algebra} $A(\mathcal{P})=A(\mathcal{P}, \textbf{k})$ is the span over $\textbf{k}$ of elements $e_{i,j}$, for $i,j\in\mathcal{P}$ satisfying $i\preceq j$, with product given by setting $e_{i,j}e_{kl}=e_{i,l}$ if $j=k$ and $0$ otherwise. The \textit{trace} of an element $\sum c_{i,j}e_{i,j}$ is $\sum c_{i,i}.$

We can equip $A(\mathcal{P})$ with the commutator product $[a,b]=ab-ba$, where juxtaposition denotes the product in $A(\mathcal{P})$, to produce the \textit{Lie poset algebra} $\mathfrak{g}(\mathcal{P})=\mathfrak{g}(\mathcal{P}, \textbf{k})$. If $|\mathcal{P}|=n$, then both $A(\mathcal{P})$ and $\mathfrak{g}(\mathcal{P})$ may be regarded as subalgebras of the algebra of $n \times n$ upper-triangular matrices over $\textbf{k}$. Such a matrix representation is realized by replacing each basis element $e_{i,j}$ by the $n\times n$ matrix $E_{i,j}$ containing a 1 in the $i,j$-entry and 0's elsewhere. The product between elements $e_{i,j}$ is then replaced by matrix multiplication between the $E_{i,j}$.

\begin{Ex}\label{ex:posetmat}
Let $\mathcal{P}$ be the poset of Example~\ref{ex:poset}. The matrix form of elements in $\mathfrak{g}(\mathcal{P})$ is illustrated in Figure~\ref{fig:tA}, where the $*$'s denote potential non-zero entries. 
\begin{figure}[H]
$$\kbordermatrix{
    & 1 & 2 & 3 & 4  \\
   1 & * & * & * & *   \\
   2 & 0 & * & * & *  \\
   3 & 0 & 0 & * & 0  \\
   4 & 0 & 0 & 0 & *  \\
  }$$
\caption{Matrix form of $\mathfrak{g}(\mathcal{P})$, for $\mathcal{P}=\{1,2,3,4\}$ with $1\preceq2\preceq3,4$}\label{fig:tA}
\end{figure}
\end{Ex}

\begin{remark}\label{rem:lpa1}
Let $\mathfrak{b}$ be the Borel subalgebra of $\mathfrak{gl}(n)$ consisting of all $n\times n$ upper-triangular matrices and $\mathfrak{h}$ be its Cartan subalgebra of diagonal matrices. Any subalgebra $\mathfrak{g}$ lying between $\mathfrak{h}$ and $\mathfrak{b}$ is then a Lie poset algebra; for $\mathfrak{g}$ is then the span over $\textbf{k}$ of $\mathfrak{h}$ and those $E_{i,j}$ which it contains, and there is a partial order on $\mathcal{P}=\{1,\dots,n\}$ compatible with the linear order by setting $i\preceq j$ whenever $E_{i,j}\in \mathfrak{g}$. 
\end{remark}

\begin{remark}\label{rem:lpa2}
Restricting $\mathfrak{g}(\mathcal{P})$ to trace-zero matrices results in a subalgebra of $A_{n-1}=\mathfrak{sl}(n)$, referred to as a type-A Lie poset algebra \textup(see \textup{\textbf{\cite{CG, CM,binary}}}\textup). As stated in the introduction, Coll and Mayers \textup{\textbf{\cite{CM}}} initiated an investigation into the index and spectral theories of type-A Lie poset algebras. 
\end{remark}

Considering Remarks~\ref{rem:lpa1} and~\ref{rem:lpa2}, we make the following definition.

\begin{definition}
Let $\mathfrak{g}$ be one of the classical simple Lie algebras. Let $\mathfrak{b}\subset\mathfrak{g}$ be the Borel subalgebra consisting of all upper-triangular matrices and $\mathfrak{h}$ be its Cartan subalgebra of diagonal matrices. A Lie subalgebra $\mathfrak{p}\subset\mathfrak{g}$ satisfying $\mathfrak{h}\subset\mathfrak{p}\subset\mathfrak{b}$ is called a Lie poset subalgebra of $\mathfrak{g}$. If $\mathfrak{g}$ is $A_{n-1}=\mathfrak{sl}(n)$, $B_n=\mathfrak{so}(2n+1)$, $C_n=\mathfrak{sp}(2n)$, or $D_n=\mathfrak{so}(2n)$, for $n\in\mathbb{N}$, then $\mathfrak{p}$ is called a type-A, type-B, type-C, or type-D Lie poset algebra, respectively.
\end{definition}

\begin{remark}\label{rem:reiner}
Let $\mathcal{P}$ be a poset with underlying set $\{1,\hdots,n\}$ which does not necessarily satisfy $\preceq_{\mathcal{P}}\subset\le$. In \textup{\textbf{\cite{Reiner1}}}, Reiner describes a method for identifying $\mathcal{P}$ with a subset of the root system corresponding to $\mathfrak{sl}(n)$. Reiner then generalizes this construction to produce what he calls a ``parset."
If $\Phi$ is a root system, then a parset \textup(partial root system\textup) is defined to be a subset $P\subset\Phi$ such that 
\begin{itemize}
    \item $\alpha\in P$ implies $-\alpha\notin P$; and
    \item if $\alpha_1,\alpha_2\in P$ and $c_1\alpha_1+c_2\alpha_2\in\Phi$ for some $c_1,c_2>0$, then $c_1\alpha_1+c_2\alpha_2\in P$.
\end{itemize}
One can attach to each type-A, B, C, and D Lie poset algebra an appropriately typed parset by pairing each Lie poset algebra with the parset generated by its roots. Such a correspondence is many-to-one, as can be seen by considering the type-B Lie poset algebras defined by the bases:
$$\{E_{1,1}-E_{5,5},E_{2,2}-E_{4,4},E_{1,2}-E_{4,5},E_{1,4}-E_{2,5}\}$$
and 
$$\{E_{1,1}-E_{5,5},E_{2,2}-E_{4,4},E_{1,2}-E_{4,5},E_{1,4}-E_{2,5},E_{1,3}-E_{3,5}\},$$ which correspond to the same parset.
\end{remark}

\section{Posets of types B, C, and D}\label{sec:BCDpos}

In this section, we provide definitions for posets of types B, C, and D (cf. \textbf{\cite{CM}}). These posets encode matrix forms that define Lie poset algebras of types B, C, and D.  

\begin{remark}
Recall that the subalgebra of upper-triangular matrices of
\begin{itemize}
    \item $\mathfrak{sp}(2n)$ consists of $2n\times 2n$ matrices of the form given in Figure~\ref{fig:matform} with $N=\widehat{N}$,
    \item $\mathfrak{so}(2n)$ consists of $2n\times 2n$ matrices of the form given in Figure~\ref{fig:matform} with $N=-\widehat{N}$, and
    \item $\mathfrak{so}(2n+1)$ consists of $2n+1\times 2n+1$ matrices of the form given in Figure~\ref{fig:matform} with $N=-\widehat{N}$ and a 0 on the diagonal separating $M$ and $-\widehat{M}$,
    \end{itemize}
    where $\widehat{N}$ denotes the transpose of $N$ with respect to the antidiagonal.
\begin{figure}[H]
$$\begin{tikzpicture}
  \matrix [matrix of math nodes,left delimiter={(},right delimiter={)}]
  {
    M  & N   \\       
    0  & -\widehat{M}   \\       };
\end{tikzpicture}$$
\caption{Matrix form}\label{fig:matform}
\end{figure}
\end{remark}

\begin{remark}\label{rem:utbasis}
Throughout the remainder of this article, unless stated otherwise, we assume that the rows and columns of a $2n\times 2n$ \textup(resp., $2n+1\times 2n+1$\textup) matrix are labeled by $\{-n,\hdots,-1,1,\hdots,n\}$ \textup(resp., $\{-n,\hdots,-1,0,1,\hdots,n\}$\textup).
\end{remark}

\begin{theorem}\label{thm:utbasisC}
A basis for a
\begin{itemize}
    \item type-C Lie poset algebra can be taken which consists solely of elements of the form $E_{-i,-i}-E_{i,i}$, for all $i\in [n]$, $E_{-i,-j}-E_{j,i}$, for $i,j\in [n]$, $E_{-i,j}+E_{-j,i}$, for $i,j\in [n]$, and $E_{-i,i}$, for $i\in [n]$.
    \item type-D Lie poset algebra can be taken which consists solely of elements of the form $E_{-i,-i}-E_{i,i}$, for all $i\in [n]$, $E_{-i,-j}-E_{j,i}$, for $i,j\in [n]$, and $E_{-i,j}-E_{-j,i}$, for $i,j\in [n]$.
    \item type-B Lie poset algebra can be taken which consists solely of elements of the form $E_{-i,-i}-E_{i,i}$, for all $i\in [n]$, $E_{-i,-j}-E_{j,i}$, for $i,j\in [n]$, $E_{-i,j}-E_{-j,i}$, for $i,j\in [n]$, and $E_{-j,0}-E_{0,j}$, for $j\in[n]$.
\end{itemize}
\end{theorem}
\begin{proof}
We prove the result for type-C Lie poset algebras as the type-B and D cases follow similarly. Let $\mathfrak{g}\subset\mathfrak{sp}(2n)$ be a type-C Lie poset algebra and $$\mathscr{B}_n=\{E_{-i,-j}-E_{j,i}~|~i,j\in [n]\}\cup\{E_{-i,j}+E_{-j,i}~|~i,j\in [n]\}\cup\{E_{-i,i}~|~i\in [n]\}.$$ We claim that if $x\in \mathscr{B}_n$ occurs as a summand with nonzero coefficient in an element of $\mathfrak{g}$, then $x\in\mathfrak{g}$. Since $\mathfrak{g}$ is a type-C Lie poset algebra, $\mathfrak{g}$ contains the Cartan subalgebra $\mathfrak{h}\subset\mathfrak{sp}(2n)$ of diagonal matrices; that is, $E_{-i,-i}-E_{i,i}\in\mathfrak{g}$, for all $i\in [n]$. Thus, letting $\mathfrak{n}\subset\mathfrak{sp}(2n)$ denote the subalgebra of strictly upper-triangular matrices, it is sufficient to show that if $a=\sum_{k=1}^rc_kx_k\in\mathfrak{g}\cap\mathfrak{n}$, where $x_k\in \mathscr{B}_n$, for $k\in[r]$, then $x_k\in\mathfrak{g}\cap\mathfrak{n}$, for some $k\in[r]$. If not, suppose that the given $a\in\mathfrak{g}$ has minimal $r$ such that no summand $x_k$ is contained in $\mathfrak{g}\cap\mathfrak{n}$; surely, $r>2$. If there exists $x_k$, a summand of $a$, such that $x_k=E_{-i,i}$, for $i\in[n]$, set $d=E_{-i,-i}-E_{i,i}$; otherwise, there exists $x_k$ such that $x_k=E_{-j,-i}-E_{i,j}$ or $E_{-j,i}+E_{-i,j}$, in which case set $d=E_{i,i}-E_{-i,-i}+E_{-j,-j}-E_{j,j}$ or $d=E_{-i,-i}-E_{i,i}+E_{-j,-j}-E_{j,j}$, respectively. In either case, $[d,a]$ is not a multiple of a but is a linear combination of the same summands $x_k$, for $k\in[r]$. To see this, note that $[d,x_k]=2x_k$, while $[d,x_l]=d_lx_l$ with $d_l=-1,0,1$, for $l\neq k$. Thus, there is a linear combination of $a$ and $[d,a]$ which is nonzero and contains no more than $r-1$ of the summands $x_k$, for $k\in[r]$; one of them is consequently already in $\mathfrak{g}\cap\mathfrak{n}$, a contradiction. Thus, since $\{E_{-i,-i}-E_{i,i}~|~i\in[n]\}\cup\mathscr{B}_n$ forms a basis of $\mathfrak{sp}(2n)$, the set $\{E_{-i,-i}-E_{i,i}~|~i\in[n]\}$ can be extended to a basis of $\mathfrak{g}$ with the desired form.
\end{proof}

\begin{definition}\label{def:BCDposet}
A type-C poset is a poset $\mathcal{P}=\{-n,\hdots,-1,1,\hdots, n\}$ such that
\begin{enumerate}
	\item if $i\preceq_{\mathcal{P}}j$, then $i\le j$; and
	\item if $i\neq -j$, then $i\preceq_{\mathcal{P}}j$ if and only if $-j\preceq_{\mathcal{P}}-i$.
\end{enumerate}
A type-D poset is a poset $\mathcal{P}=\{-n,\hdots,-1,1,\hdots, n\}$ satisfying 1 and 2 above as well as 
\begin{enumerate}
    \setcounter{enumi}{2}
    \item $i$ does not cover $-i$, for $i\in \{1,\hdots, n\}$.
\end{enumerate}
A type-B poset is a poset $\mathcal{P}=\{-n,\hdots,-1,0,1,\hdots, n\}$ satisfying 1 through 3 above.  
\end{definition}

\begin{Ex}\label{ex:CHasse}
In Figure~\ref{fig:hasse}, we illustrate the Hasse diagram of the type-C \textup(and D\textup) poset $\mathcal{P}=\{-3,-2,-1,1,2,3\}$ with $-2\preceq1,3$; $-3\preceq 2$; and $-1\preceq 2$. Note that adding 0 to $\mathcal{P}$ and a vertex labeled 0 to the Hasse diagram of Figure~\ref{fig:hasse} results in a type-B poset and its corresponding Hasse diagram.
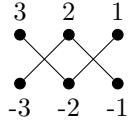
\begin{figure}[H]
$$\begin{tikzpicture}[scale = 0.65]
	\node (1) at (0, 0) [circle, draw = black, fill=black, inner sep = 0.5mm, label=below:{-3}] {};
	\node (2) at (1,0) [circle, draw = black, fill=black,  inner sep = 0.5mm, label=below:{-2}] {};
	\node (3) at (2, 0) [circle, draw = black, fill=black, inner sep = 0.5mm, label=below:{-1}] {};
    \node (4) at (0, 1) [circle, draw = black, fill=black, inner sep = 0.5mm, label=above:{3}] {};
	\node (5) at (1,1) [circle, draw = black, fill=black,  inner sep = 0.5mm, label=above:{2}] {};
	\node (6) at (2, 1) [circle, draw = black, fill=black, inner sep = 0.5mm, label=above:{1}] {};
    \draw (1)--(5);
    \draw (4)--(2);
    \draw (5)--(3);
    \draw (2)--(6);
\end{tikzpicture}$$
\caption{Hasse diagram of a type-C poset}\label{fig:hasse}
\end{figure}
\end{Ex}

\begin{theorem}
Type-C \textup(resp., B or D\textup) posets $\mathcal{P}$ are in bijective correspondence with type-C \textup(resp., B or D\textup) Lie poset algebras $\mathfrak{p}$ as follows:
\begin{itemize}
    \item $-i,i\in \mathcal{P}$ if and only if $E_{-i,-i}-E_{i,i}\in \mathfrak{p}$;
    \item $-i\preceq_{\mathcal{P}}-j$ and $j\preceq_{\mathcal{P}}i$ if and only if $E_{-i,-j}-E_{j,i}\in\mathfrak{p}$;
    \item $-i\preceq_{\mathcal{P}}j$ and $-j\preceq_{\mathcal{P}}i$ if and only if $E_{-i,j}+E_{-j,i}\in\mathfrak{p}$ \textup(resp., $E_{-i,j}-E_{-j,i}\in\mathfrak{p}$\textup);
\end{itemize}
and only in type-C
\begin{itemize}
    \item $-i\preceq_{\mathcal{P}}i$ if and only if $E_{-i,i}\in\mathfrak{p}$.
\end{itemize}
\end{theorem}
\begin{proof}
We prove the result for type-C posets and Lie poset algebras. The proofs for the types-B and D cases are similar, only requiring minor modifications. Let $\mathcal{P}$ be a type-C poset. The collection of matrices generated by the elements in $\mathfrak{p}$ clearly consists of upper-triangular matrices in $\mathfrak{sp}(|\mathcal{P}|)$ and includes the Cartan of diagonal matrices. Furthermore, transitivity in $\mathcal{P}$ guarantees closure of $\mathfrak{p}$ under the Lie bracket. Thus, $\mathfrak{p}$ is a type-C Lie poset algebra.

Conversely, if $\mathfrak{p}$ is a type-C Lie poset algebra, then taking the basis guaranteed by Theorem~\ref{thm:utbasisC}, we may use the correspondence outlined in the statement of the current theorem to form a set $\mathcal{P}$ together with a relation $\preceq_{\mathcal{P}}$ between its elements. The relations generated by the given basis elements clearly force $\preceq_{\mathcal{P}}$ to satisfy all properties required of a type-C poset except transitivity -- but transitivity is equivalent to the closure of $\mathfrak{p}$ under the Lie bracket. The result follows.
\end{proof}

\begin{remark}
Note that as in the type-A case, type-C posets $\mathcal{P}$ determine the matrix form of the corresponding type-C Lie poset algebra by identifying which entries of a $|\mathcal{P}|\times|\mathcal{P}|$ matrix can be non-zero. In particular, the $i,j$-entry can be non-zero if and only if $i\preceq_{\mathcal{P}}j$. The same is almost true in types-B and D, except one ignores relations of the form $-i\preceq_{\mathcal{P}} i$.
\end{remark}

\begin{Ex}\label{ex:typeBCD}
Let $\mathcal{P}$ be the poset of Example~\ref{ex:CHasse}. The matrix form encoded by $\mathcal{P}$ and defining the corresponding type-C \textup(and D\textup) Lie poset algebra is illustrated in Figure~\ref{fig:tBCD}, where $*$'s denote potential non-zero entries. 
\begin{figure}[H]
$$\kbordermatrix{
    & -3 & -2 & -1 & 1 & 2 & 3 \\
   -3 & * & 0 & 0 & * & * & 0  \\
   -2 & 0 & * & 0 & * & 0 & * \\
   -1 & 0 & 0 & * & 0 & * & * \\
   1 & 0 & 0 & 0 & * & 0 & 0 \\
   2 & 0 & 0 & 0 & 0 & * & 0 \\
   3 & 0 & 0 & 0 & 0 & 0 & * \\
  }$$
\caption{Matrix form for $\mathcal{P}=\{-3,-2,-1,1,2,3\}$ with $-2\preceq1,3$; $-3\preceq 2$; and $-1\preceq 2$}\label{fig:tBCD}
\end{figure}
\end{Ex}

\begin{remark}
Given a type-C poset $\mathcal{P}$, we denote the corresponding type-C Lie poset algebra by $\mathfrak{g}_C(\mathcal{P})$; furthermore, we define the following basis for $\mathfrak{g}_C(\mathcal{P})$:
\begin{align}
\mathscr{B}_C(\mathcal{P})=\{E_{-i,-i}-E_{i,i}~|~-i,i\in\mathcal{P}\}&\cup\{E_{-i,-j}-E_{j,i}~|~-i,-j,i,j\in\mathcal{P},-i\preceq -j,j\preceq i\} \nonumber \\ 
&\cup\{E_{-i,j}+E_{-j,i}~|~-i,-j,i,j\in\mathcal{P},-j\preceq i,-i\preceq j\} \nonumber \\
&\cup\{E_{-i,i}~|~-i,i\in\mathcal{P},-i\preceq i\}. \nonumber
\end{align}
Similarly, given a type-D \textup(resp., B\textup) poset $\mathcal{P}$ we denote the corresponding type-D \textup(resp., B\textup) Lie poset algebra by $\mathfrak{g}_D(\mathcal{P})$ \textup(resp., $\mathfrak{g}_B(\mathcal{P})$\textup) and define the basis $\mathscr{B}_D(\mathcal{P})$ \textup(resp., $\mathscr{B}_B(\mathcal{P})$\textup) as follows: \begin{align}
\mathscr{B}_D(\mathcal{P})=\{E_{-i,-i}-E_{i,i}~|~-i,i\in\mathcal{P}\}&\cup\{E_{-i,-j}-E_{j,i}~|~-i,-j,i,j\in\mathcal{P},-i\preceq -j,j\preceq i\} \nonumber \\
&\cup\{E_{-i,j}-E_{-j,i}~|~-i,-j,i,j\in\mathcal{P},-j\preceq i,-i\preceq j,j<i\}. \nonumber
\end{align}
\end{remark}

\begin{theorem}\label{thm:onlyC}
If $\mathcal{P}$ is a type-D poset such that $-i\npreceq i$, for all $i\in \mathcal{P}$, then $\mathfrak{g}_D(\mathcal{P})$ is isomorphic to $\mathfrak{g}_C(\mathcal{P})$.
\end{theorem}
\begin{proof}
The result follows by comparing the structure constants of $\mathfrak{g}_D(\mathcal{P})$ and $\mathfrak{g}_C(\mathcal{P})$ corresponding to the bases $\mathscr{B}_D(\mathcal{P})$ and $\mathscr{B}_C(\mathcal{P})$, respectively.
\end{proof}

\begin{theorem}\label{thm:BD}
If $\mathcal{P}$ is a type-B poset for which 0 is not related to any other element of $\mathcal{P}$ and $\mathcal{P}_0=\mathcal{P}_{\mathcal{P}\backslash\{0\}}$, then $\mathfrak{g}_B(\mathcal{P})$ is isomorphic to $\mathfrak{g}_C(\mathcal{P}_0)$.
\end{theorem}
\begin{proof}
In this case $\mathscr{B}_D(\mathcal{P}_0)$ forms a basis for both $\mathfrak{g}_B(\mathcal{P})$ and $\mathfrak{g}_D(\mathcal{P}_0)$. Applying Theorem~\ref{thm:onlyC} establishes the result.
\end{proof}

We continue to set the combinatorial notation for posets of types B, C, and D. 
\bigskip

Given a type-B, C, or D poset $\mathcal{P}$, let $\mathcal{P}^+=\mathcal{P}_{\mathcal{P}\cap\mathbb{Z}_{>0}}$ and $\mathcal{P}^-=\mathcal{P}_{\mathcal{P}\cap\mathbb{Z}_{<0}}$; that is, $\mathcal{P}^+$ (resp., $\mathcal{P}^-)$ is the poset induced by the positive (resp., negative) elements of $\mathcal{P}$.

\begin{remark}
By property 2 of Definition~\ref{def:BCDposet}, we have that $\mathcal{P}^+$ is isomorphic to $(\mathcal{P}^-)^*$. 
\end{remark}

\noindent
Let $Rel_{\pm}(\mathcal{P})$ denote the set of relations $x\preceq_{\mathcal{P}} y$ such that $x\in\mathcal{P}^-$ and $y\in\mathcal{P}^+$. We call $\mathcal{P}$ \textit{separable} if $Rel_{\pm}(\mathcal{P})=\emptyset$, and \textit{non-separable} otherwise. We say that $\mathcal{P}$ is of \textit{height} $(i,j)$ if $i$ (resp., $j$) is one less than the largest cardinality of a chain in $\mathcal{P}^+$ (resp., $\mathcal{P})$.

\begin{Ex}
If $\mathcal{P}$ is the poset of Example~\ref{ex:CHasse}, then $\mathcal{P}^+=\{1,2,3\}$ and $\mathcal{P}^-=\{-1,-2,-3\}$; both induced posets have no relations. Further, since $\mathcal{P}$ has chains of cardinality at most one, it is of height $(0,1)$.
\end{Ex}

To end this section, we introduce a condensed graphical representation for height-$(0,1)$, type-B, C, or D posets which will be used in the following section.

\begin{definition}\label{def:RG}
Given a height-$(0,1)$, type-B, C, or D poset $\mathcal{P}$, we define the relation graph $RG(\mathcal{P})$ as follows:
\begin{itemize}
    \item each pair of elements $-i,i\in\mathcal{P}$ are represented by a single vertex in $RG(\mathcal{P})$ labeled by $i\in\mathcal{P}^+$ \textup(omitting the vertex representing 0 in type B\textup);
    \item if $-i\preceq j$ in $\mathcal{P}$, then there is an edge connecting vertex $i$ and vertex $j$ in $RG(\mathcal{P})$.
\end{itemize}
If $RG(\mathcal{P})$ is connected, then $\mathcal{P}$ is called connected.
\end{definition}

\begin{remark}
Note that $RG(\mathcal{P})$ is well-defined for type-B posets $\mathcal{P}$ since if $\mathcal{P}$ is of height-$(0,1)$, then 0 cannot related to any other element of $\mathcal{P}$. Further, such relation graphs are equivalent to ``signed digraphs," as defined by Reiner \textup(see \textup{\textbf{\cite{Reiner1,Reiner2}}}\textup), with the signs removed.
\end{remark}

\begin{remark}
If $-i\preceq i$ in $\mathcal{P}$, then vertex $i$ defines a self-loop in $RG(\mathcal{P})$. Note that $RG(\mathcal{P})$ can only contain self-loops if $\mathcal{P}$ is a type-C poset.
\end{remark}

\begin{remark}
Note that normally a poset is called connected if its corresponding Hasse diagram is connected. For the purposes of this paper, though, the notion of connected given in Definition~\ref{def:RG} for type-B, C, and D posets is more useful \textup(see Theorem~\ref{thm:disjoint}\textup).
\end{remark}

\begin{Ex}
In Figure~\ref{fig:relationgr}, we illustrate the \textup(a\textup) Hasse diagram and \textup(b\textup) relation graph corresponding to the height-$(0,1)$, type-C poset $\mathcal{P}=\{-3,-2,-1,1,2,3\}$ with $-2\preceq1,2,3$; $-3\preceq 2$; and $-1\preceq 2$.
\begin{figure}[H]
$$\begin{tikzpicture}[scale = 0.65]
	\node (1) at (0, 0) [circle, draw = black, fill=black, inner sep = 0.5mm, label=below:{-3}] {};
	\node (2) at (1,0) [circle, draw = black, fill=black,  inner sep = 0.5mm, label=below:{-2}] {};
	\node (3) at (2, 0) [circle, draw = black, fill=black, inner sep = 0.5mm, label=below:{-1}] {};
    \node (4) at (0, 1) [circle, draw = black, fill=black, inner sep = 0.5mm, label=above:{3}] {};
	\node (5) at (1,1) [circle, draw = black, fill=black,  inner sep = 0.5mm, label=above:{2}] {};
	\node (6) at (2, 1) [circle, draw = black, fill=black, inner sep = 0.5mm, label=above:{1}] {};
	\node (7) at (1, -1.5) {(a)};
	\draw (5)--(2);
    \draw (1)--(5);
    \draw (4)--(2);
    \draw (5)--(3);
    \draw (2)--(6);
\end{tikzpicture}\quad\quad\begin{tikzpicture}[scale = 0.65]
	\node (1) at (0, 1.25) [circle, draw = black,fill=black, inner sep = 0.5mm, label=below:{3}] {};
	\node (2) at (1,1.25) [circle, draw = black,fill=black,  inner sep = 0.5mm, label=below:{2}] {};
	\node (3) at (2, 1.25) [circle, draw = black,fill=black, inner sep = 0.5mm, label=below:{1}] {};
	\node (7) at (1, -0.8) {(b)};
	\draw (1)--(2);
    \draw (2)--(3);
	\draw (1,1.25) .. controls (0.5,2) and (1.5,2) .. (1,1.25);
\end{tikzpicture}$$
\caption{(a) Hasse diagram and (b) relation graph of type-C poset}\label{fig:relationgr}
\end{figure}
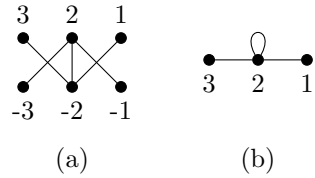
\end{Ex}

\section{Index}\label{sec:indexform}
In this section, we develop combinatorial formulas for the index of type-B,C, and D Lie poset algebras corresponding to height-$(0,1)$, type-B, C, and D posets, respectively.

It will be convenient to use an alternative characterization of the index. Let $\mathfrak{g}$ be an arbitrary Lie algebra with basis $\{x_1,...,x_n\}$.  The index of $\mathfrak{g}$ can be expressed using the \textit{commutator matrix}, $([x_i,x_j])_{1\le i, j\le n}$, over the quotient field $R(\mathfrak{g})$ of the symmetric algebra $Sym(\mathfrak{g})$ as follows (see \textbf{\cite{D}}).

\begin{theorem}\label{thm:commat}
 The index of $\mathfrak{g}$ is given by  
 
 $$\ind \mathfrak{g}= n-rank_{R(\mathfrak{g})}([x_i,x_j])_{1\le i, j\le n}.$$ 
\end{theorem}

\begin{Ex}
Consider $\mathfrak{g}_C(\mathcal{P})$, for $\mathcal{P}=\{-1,1\}$ with $-1\preceq 1$. A basis for $\mathfrak{g}_A(\mathcal{P})$ is given by $$\mathscr{B}_C(\mathcal{P})=\{x_1=E_{-1,-1}-E_{1,1},\text{ }x_2=E_{-1,1}\},$$ where $[x_1,x_2]=2x_2$. The commutator matrix $([x_i,x_j])_{1\le i, j\le 2}$ is illustrated in Figure~\ref{ex:commatsl2}.  Since the rank of this matrix is two, it follows from Theorem~\ref{thm:commat} that $\mathfrak{g}_A(\mathcal{P})$ is Frobenius.
\begin{figure}[H]
$$\begin{bmatrix}
   0 & 2x_2  \\
     -2x_2 & 0 
\end{bmatrix}$$
\caption{Commutator matrix}\label{ex:commatsl2}
\end{figure}
\end{Ex}

\begin{remark}
To ease notation, row and column labels of commutator matrices will be bolded and matrix entries will be unbolded. Furthermore, we will refer to the row corresponding to $\mathbf{x}$ in a commutator matrix -- and by a slight abuse of notation, in any equivalent matrix -- as row $\mathbf{x}$.
\end{remark}

\begin{remark}
For ease of discourse, all results of this section will be stated in terms of type-C Lie poset algebras. Considering Theorems~\ref{thm:onlyC} and~\ref{thm:BD}, all results hold in the type-B and D cases as well.
\end{remark}

Throughout this section, given a type-C poset $\mathcal{P}$, we set $$\mathcal{C}(\mathfrak{g}_C(\mathcal{P}))=([x_i,x_j])_{1\le i, j\le n}\text{, where }\{x_1,\hdots,x_n\}=\mathscr{B}_C(\mathcal{P}).$$

\begin{theorem}\label{thm:sep}
If $\mathcal{P}$ is a separable, type-C poset, then $$\ind\mathfrak{g}_C(P)=\ind\mathfrak{g}(P^+)=\ind\mathfrak{g}_A(P^+)+1.$$
\end{theorem}
\begin{proof}
Note that $$\{x_1,\hdots,x_n\}=\{E_{i,i}~|~i\in\mathcal{P}^+\}\cup\{E_{i,j}~|~i\preceq_{\mathcal{P}^+}j\}$$ and $$\{y_1,\hdots,y_n\}=\bigg\{\sum_{i=1}^{|\mathcal{P}^+|}E_{i,i}\bigg\}\cup\{E_{i,i}-E_{i+1,i+1}~|~1\le i\le |\mathcal{P}^+|-1\}\cup\{E_{i,j}~|~i\preceq_{\mathcal{P}^+}j\}$$ both form bases for $\mathfrak{g}(\mathcal{P}^+)$, while $$\{z_1,\hdots,z_{n-1}\}=\{E_{i,i}-E_{i+1,i+1}~|~1\le i\le |\mathcal{P}^+|-1\}\cup\{E_{i,j}~|~i\preceq_{\mathcal{P}^+}j\}$$ forms a basis for $\mathfrak{g}_A(\mathcal{P}^+)$. Replacing $E_{-i,-i}-E_{i,i}$ by $E_{i,i}$ and $E_{-j,-i}-E_{i,j}$ by $E_{i,j}$ in $\mathcal{C}(\mathfrak{g}_C(\mathcal{P}))$ results in $([x_i,x_j])_{1\le i, j\le n}$, establishing the first equality. The second equality follows by comparing the commutator matrices $([y_i,y_j])_{1\le i, j\le n}$ and $([z_i,z_j])_{1\le i, j\le n-1}$, for $\mathfrak{g}(\mathcal{P}^+)$ and $\mathfrak{g}_A(\mathcal{P}^+)$, respectively.
\end{proof}

\begin{corollary}
If $\mathcal{P}$ is a type-C poset such that $\mathfrak{g}_C(\mathcal{P})$ is Frobenius, then $\mathcal{P}$ is non-separable.
\end{corollary}

\begin{corollary}\label{cor:h00}
If $\mathcal{P}$ is a height-$(0,0)$, type-C poset, then $$\ind\mathfrak{g}_C(\mathcal{P})=|\mathcal{P}^+|.$$
\end{corollary} 
\begin{proof}
This follows since any commutator matrix corresponding to $\mathfrak{g}(\mathcal{P}^+)$ is the $|\mathcal{P}^+|\times|\mathcal{P}^+|$ zero-matrix.
\end{proof}

\begin{remark}  
In light of Corollary~\ref{cor:h00}, the next case to consider is type-C posets $\mathcal{P}$ of height-$(0,1)$. This case is non-trivial and contains the first examples of Frobenius, type-C Lie poset algebras. The remainder of this article concerns the analysis of this case. In particular, using combinatorial index formulas developed in Section~\ref{subsec:indexform}, we are able to fully characterize those height-$(0,1)$, type-C posets which underlie Frobenius, type-C Lie poset algebras \textup(see Theorem~\ref{thm:FrobC}\textup).  A spectral analysis follows in Section~\ref{sec:spec}.
\end{remark}

\subsection{Matrix reduction}

In this section, we describe an algorithm for reducing $\mathcal{C}(\mathfrak{g}_C(\mathcal{P}))$, for $\mathcal{P}$ a connected, height-$(0,1)$, type-C poset. This reduction facilitates the development of a combinatorial index formula for $\mathfrak{g}_C(\mathcal{P})$ in Section~\ref{subsec:indexform}.

As a first step in our matrix reduction, we order the row and column labels of $\mathcal{C}(\mathfrak{g}_C(\mathcal{P}))$, i.e., the elements of $\mathscr{B}_C(\mathcal{P})$, as follows: 
\begin{enumerate}
    \item the elements $\mathbf{E_{-i,-i}-E_{i,i}}$ in increase order of $i$ in $\mathbb{Z}$;
    \item the elements $\mathbf{E_{-i,j}+E_{-j,i}}$ in increasing lexicographic order of $(i,j)$, for $i<j$, in $\mathbb{Z}\times\mathbb{Z}$.
\end{enumerate}

\noindent
With this ordering, since height-$(0,1)$, type-C posets have no non-trivial transitivity relations, $\mathcal{C}(\mathfrak{g}_C(\mathcal{P}))$ has the form illustrated in Figure~\ref{fig:h01m}.

\begin{figure}[H]
$$\begin{tikzpicture}
  \matrix [matrix of math nodes,left delimiter={(},right delimiter={)}]
  {
    0  & -B(\mathcal{P})^T   \\       
    B(\mathcal{P})  & 0   \\       };
\end{tikzpicture}$$
\caption{Matrix form of $\mathcal{C}(\mathfrak{g}_C(\mathcal{P}))$, for $\mathcal{P}$ a height-$(0,1)$, type-C poset}\label{fig:h01m}
\end{figure}

\noindent
Here, $B(\mathcal{P})$ has rows labeled by basis elements of the form $\mathbf{E_{-i,j}+E_{-j,i}}$ and columns labeled by basis elements of the form $\mathbf{E_{-i,-i}-E_{i,i}}$, and $-B(\mathcal{P})^T$ has these labels reversed. Thus, since $rank(B(\mathcal{P}))=rank(B(\mathcal{P})^T)$, to calculate the index, it suffices to determine the rank of $B(\mathcal{P})$.

Now, in order to define our matrix reduction, we catalogue the forms of collections of rows in $B(\mathcal{P})$ which correspond to certain substructures of $RG(\mathcal{P})$. Further, we introduce row operations to reduce such collections of rows. To condense illustrations, when no confusion will arise, columns of zeros are omitted.
\\*

\noindent
\textbf{Paths}: If $RG(\mathcal{P})$ has a path consisting of the sequence of vertices $i_1,\hdots,i_n$, then the corresponding rows and columns of $B(\mathcal{P})$ have the form illustrated in Figure~\ref{fig:path}.

\begin{figure}[H]
\[
  \scalemath{0.75}{\kbordermatrix{
  & \mathbf{E_{-i_1,-i_1}-E_{i_1,i_1}} & \mathbf{E_{-i_2,-i_2}-E_{i_2,i_2}} & \mathbf{E_{-i_3,-i_3}-E_{i_3,i_3}} & \hdots & \mathbf{E_{-i_{n-1},-i_{n-1}}-E_{i_{n-1},i_{n-1}}} & \mathbf{E_{-i_n,-i_n}-E_{i_n,i_n}} \\
   \mathbf{E_{-i_1,i_2}+E_{-i_2,i_1}} &  -E_{-i_1,i_2}-E_{-i_2,i_1} & -E_{-i_1,i_2}-E_{-i_2,i_1} &  0 & \hdots & 0 & 0   \\
  \mathbf{E_{-i_2,i_3}+E_{-i_3,i_2}} &  0 & -E_{-i_2,i_3}-E_{-i_3,i_2} & -E_{-i_2,i_3}-E_{-i_3,i_2} & &     \\
  \vdots &   &  &   & \ddots & & & \\
  \mathbf{E_{-i_{n-1},i_n}+E_{-i_n,i_{n-1}}} &   &  &   & &  -E_{-i_{n-1},i_n}-E_{-i_n,i_{n-1}}  & -E_{-i_{n-1},i_n}-E_{-i_n,i_{n-1}}    \\
  }}
\]
\caption{Matrix block corresponding to a path}\label{fig:path}
\end{figure}

\begin{definition}
If $RG(\mathcal{P})$ contains a path consisting of the sequence of vertices $i_1,\hdots,i_n$, then define the row operation $Path(i_1,\hdots,i_n)$ on $\mathcal{C}(\mathfrak{g}_C(\mathcal{P}))$ to be $$(\mathbf{E_{-i_n,i_{n-1}}+E_{-i_n,i_{n-1}}})+\sum_{j=1}^{n-1}(-1)^j\frac{E_{-i_n,i_{n-1}}+E_{-i_n,i_{n-1}}}{E_{-i_{n-j},i_{n-j-1}}+E_{-i_{n-j-1},i_{n-j}}}(\mathbf{E_{-i_{n-j},i_{n-j-1}}+E_{-i_{n-j-1},i_{n-j}}})$$
performed at row $\mathbf{E_{-i_1,i_n}+E_{-i_n,i_1}}$.
\end{definition}

\begin{Ex}
The result of applying $Path(i_1,\hdots,i_n)$ to the matrix of Figure~\ref{fig:path} is illustrated in Figure~\ref{fig:rpath}.
\begin{figure}[H]
\[
  \scalemath{0.75}{\kbordermatrix{
  & \mathbf{E_{-i_1,-i_1}-E_{i_1,i_1}} & \mathbf{E_{-i_2,-i_2}-E_{i_2,i_2}} & \mathbf{E_{-i_3,-i_3}-E_{i_3,i_3}} & \hdots & \mathbf{E_{-i_{n-1},-i_{n-1}}-E_{i_{n-1},i_{n-1}}} & \mathbf{E_{-i_n,-i_n}-E_{i_n,i_n}} \\
   \mathbf{E_{-i_1,i_2}+E_{-i_2,i_1}} &  -E_{-i_1,i_2}-E_{-i_2,i_1} & -E_{-i_1,i_2}-E_{-i_2,i_1} &  0 & \hdots & 0 & 0   \\
  \mathbf{E_{-i_2,i_3}+E_{-i_3,i_2}} &  0 & -E_{-i_2,i_3}-E_{-i_3,i_2} & -E_{-i_2,i_3}-E_{-i_3,i_2} & &     \\
  \vdots &   &  &   & \ddots & & & \\
  \mathbf{E_{-i_{n-1},i_n}+E_{-i_n,i_{n-1}}} & \pm(E_{-i_{n-1},i_n}+E_{-i_n,i_{n-1}})   & 0 & 0  & \hdots &  0  & -E_{-i_{n-1},i_n}-E_{-i_n,i_{n-1}}    \\
  }}
\]
\caption{Reduced matrix block corresponding to a path}\label{fig:rpath}
\end{figure}
\end{Ex}
\bigskip

\noindent
\textbf{Self-loop}: If $RG(\mathcal{P})$ has a self-loop at vertex $i_1$, then the corresponding rows and columns of $B(\mathcal{P})$ have the form illustrated in Figure~\ref{fig:filledinnode}.

\begin{figure}[H]
\[
  \scalemath{0.75}{\kbordermatrix{
  & \mathbf{E_{-i_1,-i_1}-E_{i_1,i_1}}  \\
   \mathbf{E_{-i_1,i_1}} &  -2E_{-i_1,i_1}   \\
  }}
\]
\caption{Matrix block corresponding to a self-loop}\label{fig:filledinnode}
\end{figure}

\noindent
\textbf{Cycles}: If $RG(\mathcal{P})$ has a cycle consisting of $n>1$ vertices $i_1,\hdots,i_n$, then the corresponding rows and columns of $B(\mathcal{P})$ have the form illustrated in Figure~\ref{fig:cycle}. 

\begin{figure}[H]
\[
  \scalemath{0.75}{\kbordermatrix{
  & \mathbf{E_{-i_1,-i_1}-E_{i_1,i_1}} & \mathbf{E_{-i_2,-i_2}-E_{i_2,i_2}} & \mathbf{E_{-i_3,-i_3}-E_{i_3,i_3}} & \hdots & \mathbf{E_{-i_{n-1},-i_{n-1}}-E_{i_{n-1},i_{n-1}}} & \mathbf{E_{-i_n,-i_n}-E_{i_n,i_n}} \\
   \mathbf{E_{-i_1,i_2}+E_{-i_2,i_1}} &  -E_{-i_1,i_2}-E_{-i_2,i_1} & -E_{-i_1,i_2}-E_{-i_2,i_1} &  0 & \hdots & 0 & 0   \\
  \mathbf{E_{-i_1,i_n}+E_{-i_n,i_1}} &  -E_{-i_1,i_n}-E_{-i_n,i_1} & 0 &  0 & \hdots & 0 & -E_{-i_1,i_n}-E_{-i_n,i_1}   \\
  \mathbf{E_{-i_2,i_3}+E_{-i_3,i_2}} &  0 & -E_{-i_2,i_3}-E_{-i_3,i_2} & -E_{-i_2,i_3}-E_{-i_3,i_2} & &     \\
  \vdots &   &  &   & \ddots & & & \\
  \mathbf{E_{-i_{n-1},i_n}+E_{-i_n,i_{n-1}}} &   &  &   & &  -E_{-i_{n-1},i_n}-E_{-i_n,i_{n-1}}  & -E_{-i_{n-1},i_n}-E_{-i_n,i_{n-1}}    \\
  }}
\]
\caption{Matrix block corresponding to a cycle}\label{fig:cycle}
\end{figure}

\begin{definition}
If $RG(\mathcal{P})$ contains a cycle consisting of $n>1$ vertices $i_1,\hdots,i_n$, then define the row operation $Row_e(i_1,\hdots,i_n)$ on $\mathcal{C}(\mathfrak{g}_C(\mathcal{P}))$ to be $$(\mathbf{E_{-i_1,i_n}+E_{-i_n,i_1}})+\sum_{j=1}^{n-1}(-1)^j\frac{E_{-i_{1},i_n}+E_{-i_n,i_{1}}}{E_{-i_{j},i_{j+1}}+E_{-i_{j+1},i_j}}(\mathbf{E_{-i_{j},i_{j+1}}+E_{-i_{j+1},i_j}})$$
performed at row $\mathbf{E_{-i_1,i_n}+E_{-i_n,i_1}}$.
\end{definition}

\begin{Ex}
The result of applying $Row_e(i_1,\hdots,i_n)$ to the matrix of Figure~\ref{fig:cycle}, for $n$ even, is illustrated in Figure~\ref{fig:rcycle1}.

\begin{figure}[H]
\[
  \scalemath{0.75}{\kbordermatrix{
  & \mathbf{E_{-i_1,-i_1}-E_{i_1,i_1}} & \mathbf{E_{-i_2,-i_2}-E_{i_2,i_2}} & \mathbf{E_{-i_3,-i_3}-E_{i_3,i_3}} & \hdots & \mathbf{E_{-i_{n-1},-i_{n-1}}-E_{i_{n-1},i_{n-1}}} & \mathbf{E_{-i_n,-i_n}-E_{i_n,i_n}} \\
   \mathbf{E_{-i_1,i_2}+E_{-i_2,i_1}} &  -E_{-i_1,i_2}-E_{-i_2,i_1} & -E_{-i_1,i_2}-E_{-i_2,i_1} &  0 & \hdots & 0 & 0   \\
  \mathbf{E_{-i_1,i_n}+E_{-i_n,i_1}} &  0 & 0 &  0 & \hdots & 0 & 0   \\
  \mathbf{E_{-i_2,i_3}+E_{-i_3,i_2}} &  0 & -E_{-i_2,i_3}-E_{-i_3,i_2} & -E_{-i_2,i_3}-E_{-i_3,i_2} & &     \\
  \vdots &   &  &   & \ddots & & & \\
  \mathbf{E_{-i_{n-1},i_n}+E_{-i_n,i_{n-1}}} &   &  &   & &  -E_{-i_{n-1},i_n}-E_{-i_n,i_{n-1}}  & -E_{-i_{n-1},i_n}-E_{-i_n,i_{n-1}}    \\
  }}
\]
\caption{Reduced matrix block corresponding to an even cycle}\label{fig:rcycle1}
\end{figure}
\end{Ex}

\begin{remark}
Note that if we disregard the newly added zero row in Figure~\ref{fig:rcycle1}, then the configuration of rows is the same as that corresponding to the cycle defined by $i_1,\hdots,i_n$ with the edge between $i_1$ and $i_n$ removed.
\end{remark}

\begin{definition}
If $RG(\mathcal{P})$ contains an odd cycle consisting of $n>1$ vertices $i_1,\hdots,i_n$, then define the row operation $Row_o(i_1,\hdots,i_n)$ on $\mathcal{C}(\mathfrak{g}_C(\mathcal{P}))$ to be $Row_e(i_1,\hdots,i_n)$ followed by multiplying row $(\mathbf{E_{-i_1,i_n}+E_{-i_n,i_1}})$ by $\frac{E_{-i_n,i_n}}{E_{-i_1,i_n}+E_{-i_n,i_1}}$.
\end{definition}

\begin{Ex}
The result of applying $Row_o(i_1,\hdots,i_n)$ to the matrix of Figure~\ref{fig:cycle}, for $n$ odd, is illustrated in Figure~\ref{fig:rcycle2}.

\begin{figure}[H]
\[
  \scalemath{0.75}{\kbordermatrix{
  & \mathbf{E_{-i_1,-i_1}-E_{i_1,i_1}} & \mathbf{E_{-i_2,-i_2}-E_{i_2,i_2}} & \mathbf{E_{-i_3,-i_3}-E_{i_3,i_3}} & \hdots & \mathbf{E_{-i_{n-1},-i_{n-1}}-E_{i_{n-1},i_{n-1}}} & \mathbf{E_{-i_n,-i_n}-E_{i_n,i_n}} \\
   \mathbf{E_{-i_1,i_2}+E_{-i_2,i_1}} &  -E_{-i_1,i_2}-E_{-i_2,i_1} & -E_{-i_1,i_2}-E_{-i_2,i_1} &  0 & \hdots & 0 & 0   \\
  \mathbf{E_{-i_1,i_n}+E_{-i_n,i_1}} &  0 & 0 &  0 & \hdots & 0 & -2E_{-i_n,i_n}   \\
  \mathbf{E_{-i_2,i_3}+E_{-i_3,i_2}} &  0 & -E_{-i_2,i_3}-E_{-i_3,i_2} & -E_{-i_2,i_3}-E_{-i_3,i_2} & &     \\
  \vdots &   &  &   & \ddots & & & \\
  \mathbf{E_{-i_{n-1},i_n}+E_{-i_n,i_{n-1}}} &   &  &   & &  -E_{-i_{n-1},i_n}-E_{-i_n,i_{n-1}}  & -E_{-i_{n-1},i_n}-E_{-i_n,i_{n-1}}    \\
  }}
\]
\caption{Reduced matrix block corresponding to an odd cycle}\label{fig:rcycle2}
\end{figure}
\end{Ex}

\begin{remark}
Note that the configuration of rows in Figure~\ref{fig:rcycle2} is the same as that corresponding to the cycle defined by $i_1,\hdots,i_n$ with the edge between $i_1$ and $i_n$ removed and with $i_n$ defining a self-loop. If vertex $i_n$ already defined a self-loop, then row $\mathbf{E_{-i_1,i_n}+E_{-i_n,i_1}}$ can be reduced to a zero row in the obvious way.
\end{remark}

\noindent
Now, we come to the matrix reduction algorithm, where the relation graph $RG(\mathcal{P})$ of a connected, height-$(0,1)$, type-C poset $\mathcal{P}$ is used as a bookkeeping device to guide the reduction.
\bigskip

\begin{tcolorbox}[breakable, enhanced]
\centerline{\textbf{Matrix Reduction Algorithm}: Let $\mathcal{P}$ be a connected, height-$(0,1)$, type-C poset}
\bigskip

\noindent
\textbf{Step 1}: Set $G_1=RG(\mathcal{P})$, $M_1=B(\mathcal{P})$, $\Gamma_1=(G_1,M_1)$, and $k=1$.
\\*
\noindent
\textbf{Step 2}: Check $G_k$ for self-loops.
\begin{itemize}
    \item If $G_k$ has a self-loop at vertex $i$ and vertex $i$ is adjacent to a vertex $j$, go to \textbf{Step 3}.
    \item If $G_k$ contains self-loops and no vertex defining a self-loop is adjacent to any other vertex, halt.
    \item If $G_k$ has no self-loops, go to \textbf{Step 4}.
\end{itemize}
\textbf{Step 3}: Set $k=l$. Form $\Gamma_{l+1}=(G_{l+1},M_{l+1})$ as follows:
\begin{enumerate}
    \item Perform $$\mathbf{(E_{-i,j}+E_{-j,i})}-\frac{E_{-i,j}+E_{-j,i}}{2E_{-i,i}}\mathbf{E_{-i,i}}$$ at row $\mathbf{E_{-i,j}+E_{-j,i}}$ in $M_l$.
    \item
    \begin{itemize}
        \item If vertex $j$ does not define a self-loop, then: 
        \begin{enumerate}[label=\arabic*.]
            \setcounter{enumii}{2}
            \item Multiply row $\mathbf{E_{-i,j}+E_{-j,i}}$ by $\frac{2E_{-j,j}}{E_{-i,j}+E_{-j,i}}$ in $M_l$.
            \item Replace the row label $\mathbf{E_{-i,j}+E_{-j,i}}$ by $\mathbf{E_{-j,j}}$ in $M_l$.
            \item Remove the edge between nodes $i$ and $j$ in $G_k$ and add a self-loop at vertex $j$.
            \item Set $k=l+1$ and go to \textbf{Step 2}.
        \end{enumerate}
        \item If vertex $j$ defines a self-loop, then:
        \begin{enumerate}[label=\arabic*.]
            \setcounter{enumii}{2}
            \item Perform $$\mathbf{(E_{-i,j}+E_{-j,i})}-\frac{E_{-i,j}+E_{-j,i}}{2E_{-j,j}}\mathbf{E_{-j,j}}$$ at row $\mathbf{E_{-i,j}+E_{-j,i}}$ in $M_l$.
            \item Replace the row label $\mathbf{E_{-i,j}+E_{-j,i}}$ by $\mathbf{0}$ in $M_l$.
            \item Remove the edge between vertices $i$ and $j$ in $G_l$.
            \item Set $k=l+1$ and go to \textbf{Step 2}.
        \end{enumerate}
    \end{itemize}
\end{enumerate}

\textbf{Step 4}: Check $G_k$ for odd cycles containing $n>1$ vertices
\begin{itemize}
    \item If $G_k$ has such an odd cycle: Let $i_1,\hdots,i_n$ be the vertices of a largest odd cycle in $G_k$; if there are more than one, assuming $i_1<i_j$ for $j=2,\hdots,n$, take $(i_1,\hdots,i_n)$ to be the lexicographically least in $\mathbb{Z}^n$. Go to \textbf{Step 5}.
    \item If $G_k$ has no such odd cycles go to \textbf{Step 6}.
\end{itemize}
\textbf{Step 5:} Set $l=k$. Form $\Gamma_{l+1}=(G_{l+1},M_{l+1})$ as follows 
\begin{enumerate}
        \item Perform $Row_o(i_1,\hdots,i_n)$ in $M_l$.
        \item Replace the row label $\mathbf{E_{-i_1,i_n}+E_{-i_n,i_1}}$ by $\mathbf{E_{-i_n,i_n}}$ in $M_l$.
        \item Remove the edge between vertices $i_1$ and $i_n$, and add a self-loop at vertex $i_n$ in $G_l$.
        \item Set $k=l+1$ and go to \textbf{Step 2}
    \end{enumerate}
\textbf{Step 6:} Check $G_k$ for even cycles:
\begin{itemize}
    \item If $G_k$ has an even cycle: Let $i_1,\hdots,i_n$ be the vertices of a largest even cycle in $G_k$; if there are more than one, assuming $i_1<i_j$ for $j=2,\hdots,n$, take $(i_1,\hdots,i_n)$ to be the lexicographically least in $\mathbb{Z}^n$. Go to \textbf{Step 7}.
    \item If $G_k$ has no even cycles, go to \textbf{Step 8}.
\end{itemize}
\textbf{Step 7:} Set $l=k$. Form $\Gamma_{l+1}=(G_{l+1},M_{l+1})$ as follows
\begin{enumerate}
        \item Perform $Row_e(i_1,\hdots,i_n)$ in $M_l$.
        \item Replace the row label $\mathbf{E_{-i_1,i_n}+E_{-i_n,i_1}}$ by $\mathbf{0}$ in $M_l$.
        \item Remove the edge between vertices $i_1$ and $i_n$ in $G_l$.
        \item Set $k=l+1$ and go to \textbf{Step 2}.
    \end{enumerate}
\textbf{Step 8:} Set $l=k$. Form $\Gamma_{l+1}=(G_{l+1},M_{l+1})$ as follows

\begin{enumerate} 
    \item Take $i_1$ minimal in $\mathbb{Z}$ such that $i_1$ is a degree one vertex of $G_l$ and assume that all other vertices of $G_l$ are of maximal distance $n$ from $i_1$.
    \item Working from $m=n$ down to 1, for all paths of length $m$ in $G_l$ starting at $i_1$, say $i_1,\hdots,i_m$, perform $Path(i_1,\hdots,i_m)$ at row $\mathbf{E_{-i_m,i_{m-1}}+E_{-i_{m-1},i_m}}$.
    \item Halt.
\end{enumerate}
\end{tcolorbox}

\begin{remark}
Since we only consider finite posets $\mathcal{P}$, the above algorithm must halt in a finite number of steps.
\end{remark}

\subsection{Index formula}\label{subsec:indexform}

In this section, we determine an index formula for type-C Lie poset algebras corresponding to height-$(0,1)$, type-C posets. Throughout, we let $V(\mathcal{P})$ and $E(\mathcal{P})$ denote, respectively, the set of vertices and edges of $RG(\mathcal{P})$.

\begin{lemma}\label{lem:alg}
Let $\mathcal{P}$ be a connected, height-$(0,1)$, type-C poset. 
\begin{itemize}
    \item If $RG(\mathcal{P})$ contains an odd cycle, then the rank of $B(\mathcal{P})$ in $\mathcal{C}(\mathfrak{g}_C(\mathcal{P}))$ is $|V(\mathcal{P})|$;
    \item if $RG(\mathcal{P})$ contains no odd cycles, then the rank of $B(\mathcal{P})$ in $\mathcal{C}(\mathfrak{g}_C(\mathcal{P}))$ is $|V(\mathcal{P})|-1$;
\end{itemize}
\end{lemma}
\begin{proof}
The proof breaks into 4 cases.
\\*

\noindent
\textbf{Case 1:} If $RG(\mathcal{P})$ contains a self-loop, then the algorithm proceeds by removing adjacent edges to vertices defining self-loops and making sure, post edge removal, that such adjacent vertices define self-loops. Thus, as $RG(\mathcal{P})$ is assumed to be connected, this implies that the algorithm above halts with $\Gamma_n=(G_n,M_n)$, where 
\begin{itemize}
    \item $G_n=RG(\mathcal{P}')$ for the poset $\mathcal{P}'$ satisfying $\mathcal{P}'=\mathcal{P}$ as sets and $-i\preceq_{\mathcal{P}'}i$, for all $i\in\mathcal{P}$, and
    \item  $M_n$ is the $B(\mathcal{P}')$ block in $\mathcal{C}(\mathfrak{g}_C(\mathcal{P}'))$ with (potentially) additional zero rows;
\end{itemize}
that is, each of the $|\mathcal{P}^+|=|V(\mathcal{P})|$ columns of $M_n$ has a corresponding unique row with unique nonzero entry in that column. Thus, the result follows.
\\*

\noindent
\textbf{Case 2:} If $RG(\mathcal{P})$ contains an odd cycle consisting of $n>1$ vertices and no self-loops, then the algorithm starts by removing an edge from an odd cycle and adding a self-loop. From here, the algorithm proceeds, and the result follows, as in Case 1.
\\*

\noindent
\textbf{Case 3:} If $RG(\mathcal{P})$ is a tree, then the algorithm halts at $\Gamma_n=(G_n,M_n)$, where if $i_1$ is the specified degree one vertex of \textbf{Step 8}, then given $i_j$ and $i_k$ adjacent in $RG(\mathcal{P})$ with $i_j$ contained in the unique path from $i_1$ to $i_k$ we have that row $\mathbf{E_{-i_j,i_k}-E_{-i_k,i_j}}$ is the unique row with nonzero entry in column $\mathbf{E_{-i_k,-i_k}-E_{i_k,i_k}}$; that is, all $|\mathcal{P}^+|-1=|V(\mathcal{P})|-1$ rows of $M_n$ are linearly independent and the result follows.
\\*

\noindent
\textbf{Case 4:} If $RG(\mathcal{P})$ contains an even cycle and no odd cycles, then the algorithm removes edges from even cycles (introducing zero rows), until the resulting graph is a tree. From here, the algorithm proceeds, and the result follows, as in Case 3.
\end{proof}

As a result of Lemma~\ref{lem:alg}, we get the following.

\begin{theorem}\label{thm:indform}
If $\mathcal{P}$ is a connected, height-$(0,1)$, type-C poset, then $$\ind\mathfrak{g}_C(\mathcal{P})=|E(\mathcal{P})|-|V(\mathcal{P})|+2\delta_{\circ},$$ where $\delta_{\circ}$ is the indicator function for $RG(\mathcal{P})$ containing no odd cycles.
\end{theorem}
\begin{proof}
To start, by Theorem~\ref{thm:commat}, we know that $$\ind\mathfrak{g}_C(\mathcal{P})=dim(\mathcal{C}(\mathfrak{g}_C(\mathcal{P})))-rank(\mathcal{C}(\mathfrak{g}_C(\mathcal{P}))),$$ where $dim(\mathcal{C}(\mathfrak{g}_C(\mathcal{P})))=|E(\mathcal{P})|+|V(\mathcal{P})|$. Furthermore, $rank(\mathcal{C}(\mathfrak{g}_C(\mathcal{P})))=2\cdot rank(B(\mathcal{P}))$. By Lemma~\ref{lem:alg}, we know that if $RG(\mathcal{P})$ contains an odd cycle, then $rank(B(\mathcal{P}))=|V(\mathcal{P})|$; that is, $$\ind\mathfrak{g}_C(\mathcal{P})=|E(\mathcal{P})|+|V(\mathcal{P})|-2|V(\mathcal{P})|=|E(\mathcal{P})|-|V(\mathcal{P})|.$$ Otherwise, $rank(B(\mathcal{P}))=|V(\mathcal{P})|-1$ so that $$\ind\mathfrak{g}_C(\mathcal{P})=|E(\mathcal{P})|+|V(\mathcal{P})|-2(|V(\mathcal{P})|-1)=|E(\mathcal{P})|-|V(\mathcal{P})|+2.$$ The result follows.
\end{proof}

\begin{remark}
Note that if $RG(\mathcal{P})$ is not connected, then the elements of $\mathcal{P}$ corresponding to each connected component $K_i$ of $RG(\mathcal{P})$ induce posets $\mathcal{P}_{K_i}$ which are isomorphic to connected, type-C posets of height-$(0,0)$ or $(0,1)$.
\end{remark}

\begin{theorem}\label{thm:disjoint}
If $\mathcal{P}$ is a height-$(0,1)$, type-C poset such that $RG(\mathcal{P})$ consists of connected components $\{K_1,\hdots,K_n\}$, then $$\ind\mathfrak{g}_C(\mathcal{P})=\sum_{i=1}^n\ind\mathfrak{g}_C(\mathcal{P}_{K_i}).$$
\end{theorem}
\begin{proof}
Note that basis elements of $\mathfrak{g}_C(\mathcal{P})$ corresponding to different connected components of $RG(\mathcal{P})$ have trivial bracket relations. Thus, $\mathcal{C}(\mathfrak{g}_C(\mathcal{P}))$ can be arranged to be block diagonal with each block corresponding to the basis elements of a connected component of $RG(\mathcal{P})$. Since the block corresponding to $K_i$ is equivalent to $\mathcal{C}(\mathfrak{g}_C(\mathcal{P}_{K_i}))$, for $1\le i\le n$, the result follows.
\end{proof}

Combining Theorems~\ref{thm:indform} and~\ref{thm:disjoint} with Corollary~\ref{cor:h00} we get the following.

\begin{theorem}
If $\mathcal{P}$ is a height-$(0,1)$, type-C poset, then $$\ind\mathfrak{g}_C(\mathcal{P})=|E(\mathcal{P})|-|V(\mathcal{P})|+2\eta(\mathcal{P}),$$ where $\eta(\mathcal{P})$ denotes the number of connected components of $RG(\mathcal{P})$ containing no odd cycles.
\end{theorem}

\begin{theorem}\label{thm:FrobC}
If $\mathcal{P}$ is a height-$(0,1)$, type-C poset, then $\mathfrak{g}_C(\mathcal{P})$ is Frobenius if and only if each connected component of $RG(\mathcal{P})$ contains a single cycle which consists of an odd number of vertices.
\end{theorem}
\begin{proof}
Let $\mathcal{P}$ be a a height-$(0,1)$, type-C poset. Combining Theorem~\ref{thm:disjoint} and Corollary~\ref{cor:h00}, we find that $\mathcal{P}$ is Frobenius if and only if $\mathcal{P}$ is a disjoint sum of Frobenius, height-$(0,1)$, type-C posets. Assume $\mathcal{P}$ is connected. Note that $|E(\mathcal{P})|-|V(\mathcal{P})|\ge -1$ with equality when $RG(\mathcal{P})$ is a tree. Thus, by Theorem~\ref{thm:indform}, if $\mathcal{P}$ is Frobenius, then $RG(\mathcal{P})$ must contain an odd cycle. If $RG(\mathcal{P})$ contains an odd cycle, then $|E(\mathcal{P})|-|V(\mathcal{P})|\ge 0$ with equality if and only if $RG(\mathcal{P})$ contains a single odd cycle. Therefore, the result follows.
\end{proof}


\begin{remark}
To ease discourse in the following section, type-C \textup(resp., B or D\textup) posets corresponding to Frobenius, type-C \textup(resp., B or D\textup) Lie poset algebras are referred to as Frobenius, type-C \textup(resp., B or D\textup) posets.
\end{remark}

\section{Spectrum}\label{sec:spec}

In this section, given a Frobenius, type-B, C, or D Lie poset algebra generated by a height-$(0,1)$, type-B, C, or D poset, respectively, we determine the form of a particular Frobenius functional (see Theorem~\ref{thm:FrobFun}) as well as its corresponding principal element (see Theorem~\ref{thm:pe}). With a principal element in hand, we are then able to determine the form of the spectrum for such Lie algebras (see Theorem~\ref{thm:spec}).

\begin{remark}
As in the previous section, all results and proofs will be in terms of type-C Lie poset algebras, but all results apply to type-B and D Lie poset algebras as well.
\end{remark}

\begin{remark}
Throughout this section, we let $E^*_{i,j}$ denote the functional which returns the $i,j$-entry of a matrix.
\end{remark}

\begin{theorem}\label{thm:FrobFun}
If $\mathcal{P}$ is a Frobenius, height-$(0,1)$, type-C poset and $$F_{\mathcal{P}}=\sum_{\substack{-i\preceq j \\ i<j}}E^*_{-i,j}+\sum_{-i,i\in\mathcal{P}}\delta_{-i\preceq i}\cdot E^*_{-i,i},$$ where $\delta_{-}$ is the Kronecker delta function, then $F_\mathcal{P}$ is a Frobenius functional on $\mathfrak{g}_C(\mathcal{P})$.
\end{theorem}

\begin{remark}
Throughout this section, we will assume that if $\mathcal{P}$ is a Frobenius, height-$(0,1)$, type-C poset such that $RG(\mathcal{P})$ contains an odd cycle consisting of the vertices $\{i_1,\hdots,i_n\}$, then $i_1<i_2<\hdots<i_n$. Such an assumption does not limit the results of this section, since such a relabeling of the elements of a type-C poset does not alter the isomorphism class of the encoded type-C Lie poset algebra.
\end{remark}

\begin{lemma}\label{lem:diag}
If $\mathcal{P}$ is a Frobenius, height-$(0,1)$, type-C poset, then $B\in\mathfrak{g}_C(\mathcal{P})\cap\ker(B_{F_{\mathcal{P}}})$ must satisfy $E^*_{i,i}(B)=0$, for all $i\in\mathcal{P}$.
\end{lemma}
\begin{proof}
Given $B\in\mathfrak{g}_C(\mathcal{P})\cap\ker(B_{F_{\mathcal{P}}})$, we have 
\begin{enumerate}
    \item $F_{\mathcal{P}}(E_{-i,j}+E_{-j,i},B])=-E^*_{-i-i}(B)+E^*_{j,j}(B)=0$, for all $-j,-i,i,j\in\mathcal{P}$ satisfying $-i\preceq j$ and $i<j$; and
    \item $F_{\mathcal{P}}([E_{-i,i},B])=-E^*_{-i,-i}(B)+E^*_{i,i}(B)=0$, for $-i,i\in\mathcal{P}$ satisfying $-i\preceq i$.
\end{enumerate}
As a result of Condition 1, 
\begin{equation}\label{eqn:equal}
E^*_{-i,-i}(B)=E^*_{j,j}(B),
\end{equation}
for all $-j,-i,i,j\in\mathcal{P}$ contained in a connected component of $RG(\mathcal{P})$ satisfying $-i\preceq j$ and $i<j$. Considering each connected component $K$ of $RG(\mathcal{P})$ separately, the proof breaks into two cases.
\\*

\noindent
\textbf{Case 1:} $K$ contains a self-loop, say at vertex $i_1$. Condition 2 and the fact that $B\in\mathfrak{sp}(|\mathcal{P}|)$ combine to imply that $E^*_{i_1,i_1}(B)=E^*_{-i_1,-i_1}(B)=0$. Thus, considering (\ref{eqn:equal}) and the fact that $K$ is connected, we may conclude that $E^*_{i,i}(B)=0$ for all $i\in\mathcal{P}$ contained in $K$.
\\*

\noindent
\textbf{Case 2:} $K$ contains an odd cycle consisting of $n>1$ elements $\{i_1,\hdots,i_n\}$ satisfying: $i_1$ is adjacent to $i_{n}$ and $i_{2}$, $i_j$ is adjacent to $i_{j-1}$ and $i_{j+1}$, for $2\le j\le n-1$, and $i_1<\hdots<i_n$. Restricting Condition 1 to $\{i_1,\hdots,i_n\}$ and using the fact that $B\in\mathfrak{sp}(|\mathcal{P}|)$, i.e., $E^*_{i,i}(B)=-E^*_{-i,-i}(B)$, for all $i\in\mathcal{P}$, we find that $$E^*_{-i_1,-i_1}(B)=E^*_{i_2,i_2}(B)$$
$$E^*_{i_2,i_2}(B)=-E^*_{i_3,i_3}(B)$$ $$\hdots$$ $$E^*_{i_{n-1},i_{n-1}}(B)=-E^*_{i_n,i_n}(B)$$ $$-E^*_{i_n,i_n}(B)=E^*_{i_1,i_1}(B);$$ but, since $n$ is odd, this implies that $E^*_{-i_1,-i_1}(B)=E^*_{i_1,i_1}(B)$. The result follows as in Case 1.
\end{proof}

\begin{lemma}\label{lem:offdiag}
If $\mathcal{P}$ is a Frobenius, height-$(0,1)$, type-C poset, then $B\in\mathfrak{g}_C(\mathcal{P})\cap\ker(B_{F_{\mathcal{P}}})$ must satisfy $E^*_{-i,j}(B)=0$, for all $-i,j\in\mathcal{P}$ satisfying $-i\preceq j$.
\end{lemma}
\begin{proof}
Note that $B\in\mathfrak{g}_C(\mathcal{P})\cap\ker(B_{F_{\mathcal{P}}})$ must satisfy
\begin{eqnarray}\label{one}
F_{\mathcal{P}}([E_{-i,-i}-E_{i,i},B])=\sum_{\substack{-i\preceq j\\ i<j}}E_{-i,j}^*(B)+\sum_{\substack{-k\preceq i\\ k<i}}E^*_{-k,i}(B)+2\cdot\delta_{-i\preceq i}\cdot E^*_{-i,i}(B)=0,
\end{eqnarray}
for all $-i,i\in\mathcal{P}$. First, we show that $E^*_{-i,j}(B)=0$, for all $-i,j\in\mathcal{P}$ satisfying $-i\preceq j$, $\{j,i\}$ does not define an edge of an odd cycle in $RG(\mathcal{P})$, and $j\neq i$. Let $\Gamma_1=RG(\mathcal{P})$.
\\*

\noindent
\textbf{Step 1}: Consider all $-i,i\in \mathcal{P}$ for which $i$ is a vertex of degree one in $\Gamma_1$, say $i$ is adjacent to $j$, then $$F_{\mathcal{P}}([E_{-i,-i}-E_{i,i},B])=E^*_{-i,j}(B)=0\text{ }(\text{or }E^*_{-j,i}(B)=0).$$ Since $B\in\mathfrak{sp}(|\mathcal{P}|)$, this further implies that $$E^*_{-j,i}(B)=0\text{ }(\text{or }E^*_{-i,j}(B)=0).$$  Removing each such vertex $i$ and edge $\{i,j\}$ of $\Gamma_1$ results in $\Gamma_2$.
\\*

\noindent
\textbf{Step k}: Consider all $-i,i\in \mathcal{P}$ for which $i$ is a vertex of degree one in $\Gamma_k$, say $i$ is adjacent to $j$, then taking into account the results $\mathbf{Step\text{ }1}$ through $\mathbf{Step\text{ }k-1}$, we must have $$F_{\mathcal{P}}([E_{-i,-i}-E_{i,i},B])=E^*_{-i,j}(B)=0\text{ }(\text{or }E^*_{-j,i}(B)=0).$$ Once again, since $B\in\mathfrak{sp}(|\mathcal{P}|)$, this further implies that $$E^*_{-j,i}(B)=0\text{ }(\text{or }E^*_{-i,j}(B)=0).$$ Removing each such vertex $i$ and edge $\{i,j\}$ of $\Gamma_k$ results in $\Gamma_{k+1}$.
\\*

\noindent
Since $RG(\mathcal{P})$ is finite, there must exist $m$ for which the connected components of $\Gamma_m$ are odd cycles. Thus, $E^*_{-i,j}(B)=0$ for all $-i,j\in\mathcal{P}$ satisfying $-i\preceq j$, $\{j,i\}$ does not define an edge of an odd cycle in $RG(\mathcal{P})$, and $j\neq i$. It remains to consider $E^*_{-i,j}(B)$ corresponding to components of $\Gamma_m$. The analysis breaks into two cases.
\\*

\noindent
\textbf{Case 1:} Components consisting of a self-loop at vertex $i$. In this case, utilizing the results of $\mathbf{Step\text{ }1}$ through $\mathbf{Step\text{ }m}$ above, we must have $$F_{\mathcal{P}}([E_{-i,-i}-E_{i,i},B])=2\cdot E^*_{-i,i}(B)=0.$$ 
\\*

\noindent
\textbf{Case 2:} Components consisting of an odd cycle with $n>1$ elements $\{i_1,\hdots,i_n\}$, where $i_1$ is adjacent to $i_{n}$ and $i_{2}$, $i_j$ is adjacent to $i_{j-1}$ and $i_{j+1}$, for $2\le j\le n-1$, and $i_1<\hdots<i_n$. Restricting equation~(\ref{one}) to $\{i_1,\hdots,i_n\}$ and utilizing the results of \textbf{Step 1} through \textbf{Step m} above, we find 
$$F_{\mathcal{P}}([E_{-i_1,-i_1}-E_{i_1,i_1},B])=E^*_{-i_1,i_2}(B)+E^*_{-i_1,i_n}(B)=0$$
$$F_{\mathcal{P}}([E_{-i_2,-i_2}-E_{i_2,i_2},B])=E^*_{-i_1,i_2}(B)+E^*_{-i_2,i_3}(B)=0$$ $$\vdots$$ $$F_{\mathcal{P}}([E_{-i_{n-1},-i_{n-1}}-E_{i_{n-1},i_{n-1}},B])=E^*_{-i_{n-1},i_{n}}(B)+E^*_{-i_{n-2},i_{n-1}}(B)=0$$
$$F_{\mathcal{P}}([E_{-i_{n},-i_n}-E_{i_{n},i_{n}},B])=E^*_{-i_{1},i_{n}}(B)+E^*_{-i_{n-1},i_n}(B)=0.$$ Thus, since $n$ is odd, $$E^*_{-i_1,i_n}(B)=-E^*_{-i_1,i_2}(B)=E^*_{-i_2,i_3}(B)=\hdots=E^*_{-i_{n-1},i_{n}}(B)=-E^*_{-i_{1},i_{n}}(B);$$ that is, $E^*_{-i_1,i_n}(B)=-E^*_{-i_1,i_n}(B)=0$ and thus $E^*_{-i_j,i_{j+1}}(B)=0$, for $j=1,\hdots,n-1$. Since $B\in\mathfrak{sp}(|\mathcal{P}|)$, we also get that $E^*_{-i_j,i_{j+1}}(B)=E^*_{-i_{j+1},i_{j}}(B)=0$, for $j=1,\hdots,n-1$. The result follows.
\end{proof}

\begin{proof}[Proof of Theorem~\ref{thm:FrobFun}]
By Lemma~\ref{lem:diag} and Lemma~\ref{lem:offdiag}, if $B\in \mathfrak{g}_C(\mathcal{P})\cap\ker(B_{F_{\mathcal{P}}})$, then $B=0$.
\end{proof}

\begin{remark}
Given a poset $\mathcal{P}$ and the functional $F_{\mathcal{P}}$ as in Theorem~\ref{thm:FrobFun}, to determine the form of the principal element $\widehat{F_{\mathcal{P}}}=\sum_{i\in\mathcal{P}}c_iE_{i,i}$ note that, since $\widehat{F_{\mathcal{P}}}\in\mathfrak{sp}(|\mathcal{P}|)$, it must be the case that $(*)$ $c_i=-c_{-i}$, for all $i\in\mathcal{P}$. Furthermore, since $F_{\mathcal{P}}=F_{\mathcal{P}}\circ ad(\widehat{F_{\mathcal{P}}})$, it must be the case that $(**)$ $c_{-i}-c_j=1$, for $-i,j\in\mathcal{P}$ with $E^*_{-i,j}$ a summand of $F_{\mathcal{P}}$. It should be clear that $(*)$ and $(**)$ combine to completely characterize $\widehat{F_{\mathcal{P}}}$.
\end{remark}

\begin{theorem}\label{thm:pe}
If $\mathcal{P}$ is a Frobenius, height-$(0,1)$, type-C poset, then $\widehat{F_{\mathcal{P}}}=\sum_{i\in\mathcal{P}}c_iE_{i,i}$ satisfies 
\[c_{i} =  \begin{cases} 
      \frac{1}{2}, & i\in \mathcal{P}^+; \\
                                                &                        \\
      -\frac{1}{2}, & i\in \mathcal{P}^-.
   \end{cases}
\]
\end{theorem}
\begin{proof}
For each edge of $RG(\mathcal{P})$ between vertices $i$ and $j$ with $i<j$, we get the conditions $c_{-i}=1+c_j$ and $c_{-j}=1+c_i$. Let $K$ be a connected component of $RG(\mathcal{P})$. We claim that there exists $i\in\mathcal{P}^+$ such that $c_{-i}=1+c_i$. There are two cases.
\\*

\noindent
\textbf{Case 1}: $K$ contains a self-loop. If $K$ contains a self-loop at vertex $i$, then $c_{-i}=1+c_i$, establishing the claim.
\\*

\noindent
\textbf{Case 2}: $K$ contains an odd cycle consisting of $\{i=i_1,\hdots,i_n\}$, for $n>1$. Assume that $i_1$ is adjacent to $i_n$ and $i_2$, $i_j$ is adjacent to $i_{j-1}$ and $i_{j+1}$, for $1<j<n$, and $i_1<\hdots<i_n$. Since $n$ is odd, we must have that 
$$c_{-i}=c_{-i_1}=1+c_{i_2}=c_{-i_3}=1+c_{i_4}=\hdots=1+c_{i_{n-1}}=c_{-i_n}=1+c_{i_1}=1+c_i.$$
The claim follows.
\\*

\noindent
Now, take an arbitrary $j\in\mathcal{P}^+$ and let the sequence $j=j_0,j_1,\hdots,j_m=i$ describe a path between $j$ and $i$ in $K$. If $m$ is odd, then $$c_{-j}=c_{-j_0}=1+c_{j_1}=c_{-j_2}=\hdots=c_{-j_{m-1}}=1+c_{j_m}=1+c_{i}=c_{-i}.$$ Otherwise, $$c_{-j}=c_{-j_0}=1+c_{j_1}=c_{-j_2}=\hdots=1+c_{j_{m-1}}=c_{j_m}=c_{-i}.$$ Thus, for each connected component $K$ of $RG(\mathcal{P})$ we have that $c_{-j}=c_{-k}$ and, consequently, $c_{j}=c_{k}$, for all $j,k$ representing vertices of $K$. For $j$ representing a vertex in $K$, this implies that $c_{-j}=1+c_{j}=1-c_{-j}$; that is, $c_{-j}=\frac{1}{2}$ and $c_{j}=-\frac{1}{2}$. The result follows.
\end{proof}

\begin{theorem}\label{thm:spec}
If $\mathcal{P}$ is a Frobenius, height-$(0,1)$, type-C poset, then $\mathfrak{g}_C(\mathcal{P})$ has a spectrum consisting of an equal number of 0's and 1's.
\end{theorem}
\begin{proof}
Consider the basis $\mathscr{B}_C(\mathcal{P})$ for $\mathfrak{g}_C(\mathcal{P})$. Given the form of $\widehat{F_{\mathcal{P}}}$ found in Theorem~\ref{thm:pe}, we see that basis elements contained in the set 
$$\{E_{-i,-i}-E_{i,i}~|~-i,i\in\mathcal{P}\}$$ are eigenvectors of $ad(\widehat{F_{\mathcal{P}}})$ with eigenvalue 0, and basis elements contained in the set $$\{E_{-i,j}+E_{-j,i}~|~-i,-j,i,j\in\mathcal{P},-j\preceq i,-i\preceq j\}\cup\{E_{-i,i}~|~-i,i\in\mathcal{P},-i\preceq i\}$$  are eigenvectors of $ad(\widehat{F_{\mathcal{P}}})$ with eigenvalue 1. By Theorem~\ref{thm:FrobC}, we must have that 
\begin{align}
|\mathcal{P}^+|&=|\{E_{-i,-i}-E_{i,i}~|~-i,i\in\mathcal{P}\}| \nonumber \\
&=|\{E_{-i,j}+E_{-j,i}~|~-i,-j,i,j\in\mathcal{P},-j\preceq i,-i\preceq j\}\cup\{E_{-i,i}~|~-i,i\in\mathcal{P},-i\preceq i\}|. \nonumber
\end{align} 
Therefore, since $\mathscr{B}_C(\mathcal{P})$ is a basis for $\mathfrak{g}_C(\mathcal{P})$, the result follows.
\end{proof}

\section{Epilogue}

For type-A Lie poset algebras, there are no restrictions on the underlying poset, so the notion of ``height'' is less complicated.  In the type-A setting, the height of a poset $\mathcal{P}$ is defined to be one less than the largest cardinality of a chain. If $\mathcal{P}$ is a connected, height-one (type-A) poset, then the current authors recently established  that the index of the associated type-A Lie poset algebra $\mathfrak{g}_A(\mathcal{P})$ is given by the following nice formula (see Theorem 4, \textbf{\cite{CM}}):
\begin{eqnarray}\label{index}
\ind\mathfrak{g}_A(\mathcal{P})=|E(\mathcal{P})|-|V(\mathcal{P})|+1,
\end{eqnarray}

\noindent
where $E(\mathcal{P})$ and $V(\mathcal{P})$, are respectively, the sets of edges and vertices of the Hasse diagram of $\mathcal{P}$.  

For a connected, height-$(0,1)$ type-C poset $\mathcal{Q}$, the Hasse diagram is replaced by the relations graph $RG(\mathcal{Q})$ and the type-C analogue of (\ref{index}) is given by Theorem~\ref{thm:indform}:
\begin{eqnarray}\label{indexC}
\ind\mathfrak{g}_{C}(\mathcal{Q})=|E(\mathcal{Q})|-|V(\mathcal{Q})|+2\delta_o,
\end{eqnarray}
\noindent
where $E(\mathcal{Q})$ and $V(\mathcal{Q})$, are respectively, the sets of edges and vertices of $RG(\mathcal{Q})$, and $\delta_o$ is the indicator function for the existence of odd cycles in $RG(\mathcal{Q})$. (Of course, equation (\ref{indexC}) remains valid with ``C'' replaced by ``B'' or ``D''.)

As with Frobenius, type-C Lie poset algebras corresponding to height-$(0,1)$ posets -- but more generally -- the spectrum of $\mathfrak{g}_A(\mathcal{P})$ is binary when $\mathcal{P}$ is of height \textit{two or less}. More is known. If $\mathcal{P}$ is a \textit{toral} poset (see \textbf{\cite{binary}}) of arbitrarily height for which $\mathfrak{g}_A(\mathcal{P})$ is Frobenius, then the spectrum of $\mathfrak{g}_A$ is binary. We conjecture that having a binary spectrum is a property shared by all Frobenius Lie poset algebras in all of the classical types.


\begin{thebibliography}{abcd}


\bibitem{specD}
A. Cameron and V. Coll. ``The unbroken spectrum of Frobenius seaweeds III: type-D." \textit{Manuscript}, 2019.

\bibitem{specAB} 
A. Cameron, V. Coll, M. Hyatt, and C. Magnant. ``The unbroken spectrum of Frobenius seaweeds II: type-B and type-C." arXiv:1907.08775, July 20, 2019.

\bibitem{CG}
V. Coll and M. Gerstenhaber. 
``Cohomology of Lie semidirect products and poset algebras." \textit{J. Lie Theory}, 26: 79-95, 2016.

\bibitem{unbroken}
V. Coll, M. Hyatt, and C. Magnant. ``The unbroken spectrum of type-A Frobenius seaweeds." \textit{J. Algebraic Combinatorics}, 1-17, 2016.

\bibitem{CHM} V. Coll, M. Hyatt, and C. Magnant. ``Symplectic meanders." \textit{Comm. Algebra}, 1-13, 2017.

\bibitem{Coll2} V. Coll, M. Hyatt, C. Magnant, and H. Wang. ``Meander graphs and Frobenius seaweed Lie algebras II." \textit{J. Generalized Lie Theory and Applications}, 9(1), 2015.

\bibitem{CM}
V. Coll and N. Mayers. ``The index of Lie poset algebras." 	arXiv: 1908.06573, August 19, 2019.

\bibitem{binary}
V. Coll and N. Mayers. ``Toral posets and the binary spectrum property." 	arXiv: 1908.06573, August 19, 2019.

\bibitem{DK}
V.~Dergachev and A.~Kirillov. ``Index of Lie algebras of seaweed type." \textit{J. Lie Theory}, 10: 331--343, 2000.

\bibitem{Diatta}
A. Diatta and B. Manga. ``On properties of principal elements of Frobenius Lie algebras." \textit{J. Lie Theory}, 24: 849-864, 2014.

\bibitem{D}
J. Dixmier. 
``Enveloping Algebras." \textit{Graduate Studies in Math.} vol 2, AMS, 1996.

\bibitem{Elash} A. Elashvili. ``On the index of parabolic subalgebras of semisimple Lie algebras." \textit{Unpublished preprint}, 1990.

\bibitem{G1}
M. Gerstenhaber and A. Giaquinto. ``Boundary solutions of the classical Yang-Baxter equation." \textit{Letters Math. Physics}, 40:337-353, 1997.

\bibitem{G2}
M. Gerstenhaber and A. Giaquinto. ``Graphs, Frobenius functionals, and the classical Yang-Baxter equation." arXiv:0808.2423v1, August 18, 2008.

\bibitem{Prin}
M. Gerstenhaber and A. Giaquinto. ``The principal element of a Frobenius Lie algebra." \textit{Letters Math. Physics}, 88: 333-341, 2009.

\bibitem{Joseph} 
A. Joseph. On semi-invariants and index for biparabolic (Seaweed) Algebras, I. \textit{J. Algebra}, 305(1):487-515, 2006.

\bibitem{Ooms}
A. Ooms. ``On Frobenius Lie algebras." \textit{Comm. Algebra}, 8: 13-52, 1980.

\bibitem{Panyushev1} 
D. Panyushev. ``Inductive formulas for the index of seaweed Lie algebras." \textit{Moscow Mathematical Journal}, 1(2):221-241, 2001.

\bibitem{Panyushev2} 
D. Panyushev and O. Yakimova. ``On seaweed subalgebras and meander graphs in type C." \textit{Pacific Journal of Mathematics}, 285(2):485-499, 2016.

\bibitem{Panyushev3}
D. Panyushev and O. Yakimova. ``On seaweed subalgebras and meander graphs in type D." \textit{J. Pure and Applied Algebra}, 2018.

\bibitem{Reiner1}
V. Reiner. ``Quotients of Coxeter complexes and $P$-partitions." vol 460, AMS, 1992.

\bibitem{Reiner2}
V. Reiner. ``Signed posets." \textit{Journal of Combinatorial Theory, Series A}, 62(2): 324-360, 1993.


\end{thebibliography}
\end{document}